\documentclass[12pt]{article}
\usepackage{amsmath,amsfonts,amssymb,amsthm}
\usepackage{verbatim}
\usepackage{latexsym}
\usepackage[colorlinks,citecolor=red,urlcolor=blue,hypertexnames=false]{hyperref}

\newcommand{\note}[2][\null]{%
  \marginpar{\renewcommand{\baselinestretch}{1}\vspace{-1em}\hrule\vspace{3pt}%
  \footnotesize\raggedright\textsf{#2\ifx#1\null\else\\\hfill---
  {\em #1}\fi}\vspace{1.5em}}%
}

\newtheorem{theorem}{Theorem}[section]
\newtheorem{lemma}[theorem]{Lemma}

\newtheorem{definition}[theorem]{Definition}
\newtheorem{corollary}[theorem]{Corollary}

\newtheorem{example}[theorem]{Example}

\newcommand\RR{{\Bbb R}}
\newcommand\CC{{\Bbb C}}

\begin{document}
\title{Tiling by lattices for locally compact abelian groups}
 \author{Davide Barbieri, Eugenio Hern\'{a}ndez, Azita Mayeli
 }

 \date{\today}

\maketitle

\begin{abstract} For a locally compact  abelian group $G$ a simple proof is given for the
known fact that a bounded domain $\Omega$ tiles  $G$ with translations by a
lattice $\Lambda$ if and only if the set of characters of $G$
indexed by the dual lattice of $\Lambda$ is an orthogonal basis of
$L^2(\Omega).$ The proof uses simple techniques from Harmonic
Analysis.
\end{abstract}

\section{Introduction}

Let $G$ denote a locally compact and second countable abelian group.
A closed subgroup $\Lambda$  of $G$  is called a \emph{lattice} if
it is  discrete and  co-compact, i.e, the quotient group $G/\Lambda$
is compact. Since $G$ is second countable, then any discrete
subgroup of $G$ is also countable \cite[Section 12,
Example 17]{Pontryagin}.
We denote the dual group by $\widehat G$.  The dual lattice of
$\Lambda$ is defined as follows:
\begin{align}
\Lambda^\perp=\{ \xi\in \widehat G: \  \langle \xi,\lambda\rangle=1
\ \forall \lambda\in \Lambda\},
\end{align}
where $\langle \xi, \lambda\rangle$ indicates the action of  character $\xi$ on the group element $\lambda$.  \\

The dual lattice $\Lambda^\perp$ is a subgroup of $\widehat G$ and
by the duality theorem between subgroups and quotient groups
\cite[Lemma 2.1.2]{Rudin-FA}, it is  topologically isomorphic  to
the dual group of $G/\Lambda$, i.e., $\Lambda^\perp \cong
\widehat{G/\Lambda}$. Since $G/\Lambda$ is compact, then by the
isomorphic relation, the dual lattice $\Lambda^\perp$ is discrete.
Notice that also by \cite[Lemma 2.1.2]{Rudin-FA}, $\widehat G/\Lambda^\perp \cong \widehat\Lambda$.
This implies that $\Lambda^\perp$ is co-compact, hence it is a lattice. \\

Let $dg$ denote a Haar measure on $G$. For a function $f$ in
$L^1(G)$ the Fourier transform of $f$ is defined by
$$\mathcal F_G(f)(\chi) = \int_G f(g) \overline {\langle \chi, g\rangle}\,
dg\,, \qquad \chi\in \widehat G,$$ where $\langle \chi, g\rangle$
denotes the action of the character $\chi$ on $g.$ By the inversion
theorem \cite[Section 1.5.1]{Rudin-FA}, a Haar measure $d\chi$ can
be chosen on $\widehat G$ so that the Fourier transform $\mathcal
F_G$ is an isometry from $L^2(G)$ onto
$L^2(\widehat G).$ \\

For any $\chi\in \widehat G$, we define the {\it exponential
function} $e_{\chi}$ by
  $$e_{\chi}: G\to \CC, \quad \quad e_{\chi}(g):= \langle \chi, g\rangle. $$
For any subset   $\Omega$ of $G$, we let $|\Omega|$ denote the Haar
measure of $\Omega$. Throughout this paper, we let ${\bf 1}_\Omega$
denote the characteristic function of  set $\Omega$.  We shall also
use addition $`+\rq{}$ for the group action since $G$ is abelian.

\begin{definition}[Tiling]   We say a subset  $\Omega\subset G$   with positive and non-zero Haar measure, tiles
  $G$ with the  translations by a  lattice $\Lambda$ of $G$ if
   for almost every  $g\in G$ there is only one $\lambda\in \Lambda$ and only one $x\in \Omega$ such that $g$
   has the representation $g= x+\lambda$.  Notice that the uniqueness of the representations  implies that
   for any two lattice points  $\lambda_1\neq \lambda_2$,  the  sets  $\Omega+\lambda_1$ and $ \Omega+\lambda_2$
   do not intersect up to a measure zero set. In this case we say $(\Omega, \Lambda)$ is  a {\it tiling pair}.
   Equivalently, $(\Omega, \Lambda)$ is a tiling pair if

\begin{align}\label{tiling-property}
\sum_{\lambda\in\Lambda} {\bf 1}_\Omega(g-\lambda) = 1 \quad a.e. \
g\in  G\,.
\end{align}
\end{definition}

\begin{definition}[Spectrum] Let $\widetilde \Lambda$ be a countable subset
of $\widehat G$. We say $\widetilde \Lambda$ is a spectrum for
$\Omega$, if the exponentials $\{ e_{\tilde\lambda}:
\tilde\lambda\in \widetilde \Lambda\}$ form an orthogonal basis for
$L^2(\Omega)$.
\end{definition}

In the sequel, we  will assume that $\Omega$ has positive measure
and $\Lambda$ is a lattice in $G.$ Therefore, the annihilator
$\Lambda^\perp$ is also a lattice in $\widehat G$. Our main  result
is the following.

 \begin{theorem}[Main Result]\label{main} Let $\Lambda$ and $\Omega$ be as above. Then the following are
 equivalent:
 \begin{enumerate}
 \item[(1)]\label{(1)} The set $\Omega$ tiles $G$ with translations by the lattice $\Lambda$.
 \item[(2)]\label{(2)} $|\Omega| = |Q_\Lambda|$, and the system of translations $\{\sqrt{|\Omega|}^{~-1}{\bf
1}_\Omega (\cdot-\lambda):
  \ \lambda\in \Lambda\}$ is an orthonormal system in $L^2(G)$.
  \item[(3)]\label{(3)} $|\Omega| = |Q_\Lambda|$, and $$\sum_{\tilde\lambda\in \Lambda^\perp} |\mathcal F_G({\bf 1}_\Omega) (\chi+\tilde\lambda)|^2=|\Omega|^2
 \quad a.e. \ \ \chi\in \widehat G\,.$$
   \item[(4)]\label{(4)} For all $f\in L^2(G),$
$$\|f{\bf 1}_\Omega\|^2_{L^2(G)} =|\Omega|^{-1}  \sum_{\tilde\lambda\in \Lambda^\perp}
|\mathcal F_G(f{\bf 1}_\Omega)(\tilde\lambda)|^2 \,.$$
    \item[(5)]\label{(5)} The exponential set $\{e_{\tilde\lambda}: \tilde\lambda\in\Lambda^\perp\}$ is an orthogonal basis for $L^2(\Omega)$.
\end{enumerate}
 \end{theorem}

The study of the relationship between spectrum sets and tiling pairs
has its roots in a 1974 paper of B. Fuglede (\cite{Fue74}) where he
proved that a set $E \subset {\Bbb R}^d$, $d\geq 1$, of positive
Lebesgue measure, tiles $\Bbb R^d$ by translations with a lattice
$\Lambda$ if and only if $L^2(E)$ has an orthogonal basis of
exponentials indexed by the annihilator of $\Lambda.$ \\

The general statement in ${\Bbb R}^d$, which says that if $E \subset
{\Bbb R}^d$, $d\geq 1$, has positive Lebesgue measure, then $L^2(E)$
has an orthogonal basis of exponentials (not necessary indexed by a
lattice) if and only if $E$ tiles ${\Bbb R}^d$ by translations, has
been known as the Fuglede Conjecture. A variety of results were
proved establishing connections between tiling and orthogonal
exponential bases. See, for example, \cite{LRW00}, \cite{IP98},
\cite{L02}, \cite{KL03} and \cite{KL04}. In 2001, I. Laba proved
that the Fuglede conjecture is true for the union of two intervals
in the plane (\cite{L01}). In 2003, A. Iosevich, N. Katz and T. Tao
(\cite{IKT03}) proved that the Fuglede conjecture holds for convex
planar domains. The conjecture was also proved for the unit cube of
$\Bbb R^d$ in \cite{IP98} and \cite{LRW00}.
\\

In 2004 Terry Tao (\cite{T04}) disproved the Fuglede Conjecture in
dimension $d=5$ and larger, by exhibiting a spectral set in ${\Bbb
R}^{5}$ which does not tile the space by translations. In
\cite{KM06}, M. Kolountzakis and M. Matolcsi also disproved the
reverse implication of the Fuglede Conjecture for dimensions $d=4$
and larger. In \cite{FR06} and \cite{FMM06}, the dimension of
counter-examples was further reduced. In fact, B. Farkas, M.
Matolcsi and P. Mora show in \cite{FMM06} that the Fuglede
conjecture is false in $\RR^3$. The general feeling in the field is
that sooner or later the counter-examples of both implications will
cover all dimensions. However, in \cite{IMP15} the authors  showed
that the Fuglede
Conjecture holds in two-dimensional vector spaces over prime fields.  \\

The extension of B. Fuglede result for lattices in $\Bbb R^d, d\geq
1,$ to second countable LCA groups has been proved by S. Pedersen in
\cite{Pedersen}, where mainly topological arguments are used to
prove that $(1)$ and $(5)$ of Theorem \ref{main} are equivalent. In
the current paper, analytical methods are used instead to prove this
result, while other equivalent conditions are given. We emphasize
that the equivalence of $(2)$ and $(3)$ in Theorem \ref{main} comes
from the theory of dual integrable representations developed in
\cite{HSWW} (see Section \ref{Pre} for details). The authors
strongly believe  that the techniques used in this paper are
suggestive and can lead to more discoveries on the hidden relations
between translations and exponential bases in general.   \\

{\bf Acknowledgement}: D. Barbieri was supported by a Marie Curie
Intra European Fellowship (Prop. N. 626055) within the 7th European
Community Framework Programme.
 D. Barbieri and E. Hern\'andez were supported by Grant MTM2013-40945-P (Ministerio de Econom\'ia y Competitividad, Spain).
A. Mayeli was supported by PSC-CUNY-TRADB-45-446, and by the
Postgraduate Program of Excellence in Mathematics at Universidad
Aut?oma de Madrid from June 19 to July 17, 2014, when this paper
was completed. The authors wish to thank Alex Iosevich for several
interesting conversations regarding this paper. The proof of the
main result is partly influenced by his presentation of the proof of
the Fuglede conjecture for lattice (presented in CANT-CUNY
Conference 15) and his expository paper on this subject
\cite{Iosevich-Expository}

\subsection{Notations and Preliminaries} \label{Pre}

Let  $\Lambda$ be a lattice in a second countable LCA group $G$.
Denote by $Q_\Lambda \subset G$ a measurable cross section of
$G/\Lambda.$  By definition, a cross section is a set of
representatives of all cosets in $G/\Lambda $ such that the
intersection of $Q_\Lambda$ with any coset $g+\Lambda$ has only one
element. Its existence is guaranteed by \cite[Theorem 1]{FG68}. A
cross section is called {\it fundamental domain}  if it is
relatively compact. In our situation, every cross section is a
fundamental domain since $G/\Lambda$ is compact. By the definition,
it is evident that ($Q_\Lambda$, $\Lambda$) is a tiling pair for
$G$. For a lattice $\Lambda$, the {\it lattice size} is defined as
the Haar measure of a fundamental domain
$Q_\Lambda$, i.e., $|Q_\Lambda|$, therefore the measure of any fundamental domain. \\

Let $d\dot g$ be a normalized Haar measure for $G/\Lambda$. Then the
relation between  Haar measure on $G$ and Haar measure for
$G/\Lambda$ is given by the following {\it Weil\rq{}s formula}: for
any function $f\in L^1(G)$,  the periodization map $\Phi(\dot
g)=\sum_{\lambda\in \Lambda} f(g+\lambda), \ \dot g\in G/\Lambda$ is
well defined almost everywhere in $G/\Lambda$, belongs to
$L^1(G/\Lambda),$ and

\begin{align} \label{Weil}
\int_G f(g)dg = |Q_\Lambda| \int_{G/\Lambda} \sum_{\lambda\in\Lambda}  f(g+\lambda) d\dot g.
\end{align}
This formula is a especial case of \cite[Theorem 3.4.6]{RS68}. The
constant $|Q_\Lambda|$ appears in \ref{Weil} because $G/\Lambda$
is equipped with the normalized Haar measure. \\

Here, we shall recall {\it Poisson summation formula} (see \cite[Lemma
6.2.2]{KG} or \cite[Theorem 4.42]{Folland-book}).

\begin{theorem}\label{Poisson-formula-theorem}
Given $f\in L^1(G)$ and a lattice $\Lambda$ in $G$, if $\sum_{\tilde
\lambda\in \Lambda^\perp} |\mathcal F_G
(f)(\tilde\lambda)|^2<\infty$, then the periodization map $\Phi$,
given before (\ref{Weil}), belongs to  $L^2(G/\Lambda)$ and
\begin{align}\label{Poisson-summation-formula}
\sum_{\lambda\in \Lambda} f(g+\lambda) = |Q_\Lambda|^{-1}
\sum_{\tilde\lambda\in \Lambda^\perp} \mathcal F_G
(f)(\tilde\lambda) e_{\tilde\lambda}(g)\quad a.e. \ g\in G.
\end{align}
Here, both series converge in $L^2(G/\Lambda)$. Moreover, if $f\in
C_c(G)$, i.e, continuous and compactly supported, then Poisson
Summation formula (\ref{Poisson-summation-formula})  holds pointwise
(\cite{Folland-book}).
   \end{theorem}

For any function $f:G\to \CC$, we define $\widetilde f(g) :=
\overline{f(-g)}$. And for any given two functions $f$ and $h$
defined on $G$, the convolution $f\ast h$ is given by
 $$f\ast h(g) = \int_G f(g_1) h(g-g_1) dg_1,$$
 provided that the integral exists.

\begin{definition}[Dual integrable representations (\cite{HSWW})]
Let $G$ be an LCA group, and  let $\pi$ be a unitary representation
of $G$ on a Hilbert space $\mathcal H$ with inner product $\langle \
, \ \rangle$. We say $\pi$ is {\it dual integrable} if there exists
a Haar measure $d\chi$ on $\widehat G$ and a function $[\cdot,
\cdot]_\pi: \mathcal H\times \mathcal H\to L^1(\widehat G)$, called
{\it bracket map} for $\pi,$ such that for all $  \phi, \psi\in
\mathcal H$
$$\langle \phi, \pi(g)\psi\rangle =
\int_{\widehat G} [\phi, \psi]_\pi(\chi) e_{-g}(\chi) d\chi \quad \ \  \forall \ g\in G.
$$
\end{definition}

\begin{example}\label{translation}
Let $\Lambda$ be a lattice in a second countable LCA group $G$. Let
$\phi\in L^2(G)$.  For any $\lambda\in \Lambda$, define
$T_\lambda(\phi)(g)= \phi(g-\lambda)$. We show that $T$ is a dual
integrable representation. Let $M_\lambda h(\chi):=
e_{\lambda}(\chi) h(\chi)$, and let $Q_{\Lambda^\perp}$ be a
fundamental domain for the annihilator lattice $\Lambda^\perp$.
Thus, by an application of Parseval identity and Weil\rq{}s formula
(\ref{Weil}) we have

\begin{align} \langle \phi, T_\gamma\psi\rangle &= \langle \mathcal F_G (\phi) , M_\gamma \mathcal F_G (\psi)\rangle\\\notag
& = \int_{\widehat G} \mathcal F_G (\phi)(\chi) \overline{\mathcal
F_G (\psi)(\chi)} e_{-\gamma}(\chi) d\chi \\\notag
&=|Q_{\Lambda^\perp}| \int_{\widehat G/{\Lambda^\perp}}
\sum_{\tilde\lambda\in\Lambda^\perp} \mathcal F_G (\phi)(\chi+\tilde
\lambda) \overline{\mathcal F_G (\psi)(\chi+\tilde \lambda)}
e_{-\lambda}(\chi+\tilde\lambda) d\dot\chi\\\notag &=
|Q_{\Lambda^\perp}| \int_{\widehat G/{\Lambda^\perp}}
\sum_{\tilde\lambda\in\Lambda^\perp} \mathcal F_G (\phi)(\chi+\tilde
\lambda) \overline{\mathcal F_G (\psi)(\chi+\tilde \lambda)}
e_{-\lambda}(\chi) d\dot\chi.
 \end{align}

 Now take

 $$[\phi, \psi]_T(\dot\chi): = |Q_{\Lambda^\perp}| \sum_{\tilde \lambda\in \Lambda^\perp}
 \mathcal F_G (\phi)(\chi+\tilde\lambda)\overline{\mathcal F_G
 (\psi)(\chi+\tilde\lambda)}\,
 \qquad a. e. \quad \dot\chi\in \widehat G/{\Lambda^\perp} \ .$$

 Then, $T$ is a dual integrable representation with $\mathcal H=L^2(G)$ and the bracket function $[\cdot ,
 \cdot]_T$, which belongs to $L^1(\widehat G/\Lambda^\perp)$ because
 $ \mathcal F_G(\phi) \overline{\mathcal F_G (\psi)} \in L^1(\widehat G).$
\end{example}

One of practical application of dual integrable representations is
that one can characterize bases in terms of their associated bracket
functions. For example, the following result has been given in \cite[Proposition 5.1]{HSWW}.

 \begin{theorem}\label{Result in HSWW}
 Let $G$ be a countable abelian group and $\pi$  be a dual
integrable unitary representation of $G$ on a Hilbert space
$\mathcal H$. Let $\phi\in \mathcal H$. Then the system
$\{\pi(g)\phi: g\in G\}$ is an orthonormal basis for its spanned
vector space if and only if $[\phi, \phi]_\pi(\chi)= 1$ for almost
every $\chi\in \widehat G$.
 \end{theorem}

 As a  byproduct of the preceding theorem we have the following result.

\begin{corollary}\label{orthonormal translation system}
Let $T$ and $\Lambda$ be   as in Example   \ref{translation}.   Let
$\phi\in L^2(G)$. Then   the translations system $\{T_\lambda \phi:
\lambda\in \Lambda\}$ is an orthonormal system in $L^2(G)$ if and
only if
 $$\sum_{\tilde\lambda\in \Lambda^\perp} |\mathcal F_G(\phi) (\chi+\tilde\lambda)|^2=
 |Q_{\Lambda^\perp}|^{-1} \quad  a.e. \ \ \chi\in \widehat
 G\,.
 $$
\end{corollary}

We need a couple of Lemmata that will also be used in the proof of
Theorem \ref{main}.

\begin{lemma}\label{measure of Omega}
Let $\Omega$ tile $G$ by lattice $\Lambda$. Let $Q_\Lambda$ denote a
fundamental domain for  $\Lambda$. Then $|\Omega|=|Q_\Lambda|$.
\end{lemma}

\begin{proof}   Let $f= {\bf 1}_\Omega$ be the characteristic function of  $\Omega$.
 By Weil\rq{}s formula (\ref{Weil}) we have:

 \begin{align} |\Omega| = \int_G {\bf 1}_\Omega(g) dg =
 |Q_{\Lambda}| \int_{G/\Lambda}  \sum_{\lambda\in \Lambda}{\bf 1}_\Omega(g+\gamma) d\dot g
 \end{align}
By the assumption, $\Omega$ tiles $G$ by $\Lambda$-translations.
Therefore,

 $$\sum_{\lambda\in \Lambda}{\bf 1}_\Omega(g+\gamma) =1, \quad  a.e.  g\in G. $$
  Thus
 \begin{align} |\Omega| =  |Q_{\Lambda}|  \int_{G/\Lambda}    1 d\dot g = |Q_\Lambda|,
 \end{align}
 since $d\dot g$ is a normalized Haar measure for $G/\Lambda$. This proves our  lemma.
 \end{proof}

 We shall recall the following lemma whose proof has been shown in \cite[Lemma 6.2.3]{KG}.

 \begin{lemma} \label{product}
 Let $\Lambda$ be a lattice in  $G$. Then the annihilator $\Lambda^\perp$ is a
 lattice in $\widehat G$ and $$|Q_\Lambda| |Q_{\Lambda^\perp}|=1 ~.$$
 \end{lemma}

\section{Proof of main Theorem - Theorem \ref{main}}

  In the rest, we shall prove Theorem \ref{main}.
 \begin{proof}

``(1)$\Leftrightarrow$(2)\rq\rq{}. The implication
``(1)$\Rightarrow$(2)\rq\rq{} is trivial since $\Omega$ tiles $G$ by
$\Lambda$-translations. We now prove ``(2)$\Rightarrow$(1)\rq\rq{}.
By the orthogonality of translations,
  for any distinct $\lambda_1  $ and  $\lambda_2\in \Lambda,$
$$
0=\langle {\bf 1}_\Omega(\cdot-\lambda_1),
{\bf 1}_\Omega(\cdot-\lambda_2)\rangle_{L^2(G)}= |\Omega+\lambda_1\cap \Omega+\lambda_2| .
$$
This indicates that the  translations of $\Omega$ by  $\lambda_1\neq
\lambda_2$ have  an intersection whose measure is zero. It remains
to show that the all $\lambda$-translations of $\Omega$, $\lambda\in
\Lambda$, cover the whole group $G$ almost everywhere. Let
$F:=\dot\cup_{\lambda\in \Lambda} \Omega+\lambda$, where the
$\dot\cup$ is used to denote almost disjoint union. Then
$$\sum_{\lambda\in \Lambda} {\bf 1}_\Omega(x-\lambda)={\bf 1}_F(x)\quad a.e. \ x\in F .$$

Let  $E=G\setminus F$. We shall prove that $|E|=0$. Let $Q_\Lambda$
be any fundamental set for $\Lambda$. We prove that $|F\cap
Q_\Lambda| = |Q_\Lambda|$:

\begin{align}
|F\cap Q_\Lambda|= \int_G {\bf 1}_{F\cap Q_\Lambda} (x) dx &= \int_G {\bf 1}_{Q_\Lambda}(x) {\bf 1}_F(x) dx\\\notag
 &= \int_G {\bf 1}_{Q_\Lambda}(x) \left( \sum_{\lambda\in \Lambda} {\bf 1}_\Omega(x-\lambda)
 \right) dx\\\notag
 &= \sum_{\lambda\in \Lambda} \int_{Q_\Lambda} {\bf 1}_\Omega(x-\lambda) dx\\\notag
 &= \sum_{\lambda\in \Lambda} \int_{\lambda+Q_{\Lambda}}  {\bf 1}_\Omega(y)dy\\\notag
 &= \int_G  {\bf 1}_\Omega(y)dy = |\Omega| = |Q_\Lambda| .
\end{align}

From the above calculations, we conclude that  $F\cap Q_\Lambda =
Q_\Lambda$ up to a set of measure zero, since $Q_\Lambda$ has finite
measure. Similarly, one can show that for all $\lambda\in \Lambda$,
$F\cap (Q_\Lambda+\lambda) = Q_\Lambda+\lambda$ up to a measure zero
set.  Since $G$ is a countable union of the pairwise disjoint sets
$Q_\Lambda + \lambda, \lambda \in \Lambda,$ the set $E= G\setminus
F$ is a disjoint countable union of sets of measure zero. Hence, $E$
is
a set of measure zero. This completes the proof. \\

``(2)$\Leftrightarrow$(3)\rq\rq{}. Follows from Corollary
\ref{orthonormal translation system}
 and Lemmata \ref{measure of Omega} and \ref{product}.  \\

``(4) $\Leftrightarrow$ (5)\rq\rq{}. The implication ``(5)
$\Rightarrow$ (4)\rq\rq{} is Parseval identity for the exponential
basis $\{e_{\tilde\lambda}{\bf 1}_\Omega : \tilde \lambda \in
\Lambda^\perp \}$. Let us prove now ``(4) $\Rightarrow$ (5)\rq\rq{}.
Let $h\in L^2(\Omega)$  be such that $\langle h , e_{\tilde \lambda}
{\bf 1}_\Omega\rangle=0$ for all $\tilde\lambda\in \Lambda^\perp$.
Then
$$
\mathcal F_G(h{\bf 1}_\Omega)(\tilde\lambda)=\int_\Omega  h(g)
\overline{e_{\tilde \lambda}(g)} dg   = \langle h ,  e_{\tilde
\lambda}{\bf 1}_\Omega\rangle = 0\,.
$$
Therefore, by  the  equation in condition (4), $h=0$
a.e..  This proves the completeness of the exponentials $\{
e_{\tilde \lambda}\}_{\Lambda^\perp}$ in $L^2(\Omega)$.  To show
that the exponentials are orthogonal,  fix $\tilde\lambda_0\in
\Lambda^\perp$ and put  $f(g) := e_{\tilde\lambda_0}(g){\bf
1_\Omega}(g)$.  Once again, by condition (4) we have

\begin{align}\notag
|\Omega|= \|f\|^2_{L^2(G)}&=|\Omega|^{-1} \sum_{\tilde\lambda\in \Lambda^\perp} |\mathcal F_G(e_{\tilde\lambda_0}{\bf 1}_\Omega)(\tilde\lambda)|^2\\\notag
&=  |\Omega|^{-1}
\sum_{\tilde\lambda\in \Lambda^\perp} |\langle e_{\tilde\lambda_0}, e_{\tilde\lambda}\rangle_{L^2(\Omega)}|^2 \\\notag
&= |\Omega|   +
 |\Omega|^{-1}\sum_{\tilde\lambda\neq \tilde\lambda_0} |\langle e_{\tilde\lambda_0}, e_{\tilde\lambda}\rangle_{L^2(\Omega)}|^2.
\end{align}
 This implies
 that
  $$\langle e_{\tilde\lambda_0}, e_{\tilde\lambda}\rangle_{L^2(\Omega)}=0  \quad \quad \forall \  \tilde\lambda\neq \tilde\lambda_0,  $$
and this takes care of the orthogonality of the exponentials. \\

``(4)$\Rightarrow$(3)\rq\rq{}. Fix  $\chi\in \widehat G$ and define
$f(g) = \overline{\chi(g)}   {\bf 1}_\Omega(g)$,  $g\in G$. Then  we
have
$$
\mathcal F_G(f)(\tilde \lambda) = \mathcal F_G({\bf 1}_\Omega)(\tilde \lambda+
  \chi), \quad  \forall \ \tilde\lambda\in \Lambda^\perp\,.
$$
Thus  for $f$, by the condition (4), we obtain
$$
\sum_{\tilde\lambda\in \Lambda^\perp} |\mathcal F_G({\bf 1}_\Omega)(\tilde \lambda+
  \chi)|^2=
   |\Omega| \|\overline{\chi(g)}{\bf 1}_\Omega\|^2 = |\Omega|^2 ~.
  $$
\\

``(1)$\Rightarrow$(4)\rq\rq{}. Let  $f\in C_c(G)$, be continuous and
compactly supported. We can write the following:

\begin{align}\notag 
\sum_{\tilde\lambda\in \Lambda^\perp} |\mathcal F_G(f{\bf
1}_\Omega)(\tilde\lambda)|^2 &=  \sum_{\tilde\lambda\in
\Lambda^\perp}  \mathcal F_G(f{\bf 1}_\Omega)(\tilde\lambda)
\overline{\mathcal F_G(f{\bf 1}_\Omega(\tilde\lambda) }\\\notag &=
\sum_{\tilde\lambda\in \Lambda^\perp}  \mathcal F_G(f{\bf
1}_\Omega)(\tilde\lambda) \mathcal F_G(\widetilde{f{\bf
1}_\Omega})(\tilde\lambda)\\\notag &= \sum_{\tilde\lambda\in
\Lambda^\perp}  \mathcal F_G(f{\bf 1}_\Omega\ast  \widetilde{f{\bf
1}_\Omega})(\tilde\lambda) .
\end{align}

Since $f\in C_c(G)$, then   $f{\bf 1}_\Omega\ast \widetilde{f{\bf
1}_\Omega} \in C_c(G)$. Then  Poisson summation  formula (see
Theorem \ref{Poisson-formula-theorem}) for $f{\bf 1}_\Omega\ast
f{\bf 1}_\Omega$ at $g=e$ implies that
$$
\sum_{\tilde\lambda\in \Lambda^\perp}  \mathcal F_G(f{\bf 1}_\Omega\ast
\widetilde{f{\bf 1}_\Omega})(\tilde\lambda) = |\Omega| \sum_{\lambda\in \Lambda}
(f{\bf 1}_\Omega\ast  \widetilde{f{\bf 1}_\Omega})(\lambda) ~.$$

Therefore,
\begin{align}\notag
\sum_{\tilde\lambda\in \Lambda^\perp} |\mathcal F_G(f{\bf
1}_\Omega)(\tilde\lambda)|^2
    &=|\Omega| \sum_{\lambda\in \Lambda}   (f{\bf 1}_\Omega\ast  \widetilde{f{\bf 1}_\Omega})(\lambda) \\\notag
        &=  |\Omega|\sum_{\lambda\in \Lambda}    \int_G (f{\bf 1}_\Omega)(g)  \widetilde{(f{\bf 1}_\Omega)}(\lambda-g)dg \\\notag
      &= |\Omega| \sum_{\lambda\in \Lambda}    \int_G (f{\bf 1}_\Omega)(g)  \overline{(f{\bf 1}_\Omega)(g-\lambda)} dg \\\notag
      &= |\Omega| \sum_{\lambda\in \Lambda}    \int_{\Omega\cap \Omega+\lambda}   f(g)  \overline{f(g-\lambda)} dg\,. \\\notag
\end{align}
By the assumption (1) of the theorem,  $|\Omega\cap
\Omega+\lambda|=0$ for $\lambda\neq e$. Therefore,

\begin{align}\notag
  |\Omega|^{-1}  \sum_{\tilde\lambda\in \Lambda^\perp} |\mathcal F_G(f{\bf 1}_\Omega)(\tilde\lambda)|^2 =
   \int_{\Omega}   f(g)  \overline{f}(g) dg
  =
   \int_{G}   |f{\bf 1}_\Omega(g)|^2  dg =  \|f{\bf 1}_\Omega\|_{L^2(G)}^2
   ~.
\end{align}
This shows (4) for $f\in C_c(G)$. Use a density argument to prove
the result for general $f\in L^2(G)$.

\end{proof}

\

\textsf{D. Barbieri}, Universidad Aut\'onoma de Madrid, 28049
Madrid, Spain.

\textsf{davide.barbieri@uam.es}

\

\textsf{E. Hern\'andez},  Universidad Aut\'onoma de Madrid, 28049
Madrid, Spain.

\textsf{eugenio.hernandez@uam.es}

\

\textsf{A. Mayeli}, City University of New York, Queensborough.

 \textsf{amayeli@gc.cuny.edu}


\begin{thebibliography}{20}
   \bibitem{Ole-Christensen-Book}
 O. Christensen, {\it An Introduction to Frames and Riesz Bases},  Brikh\"{a}user, (2003)

 \bibitem{FMM06} B. Farkas, M. Matolcsi and P. M\'{o}ra,
 {\it On Fuglede's conjecture and the existence of universal spectra},
 J. Fourier Anal. Appl. \textbf{12} (2006), no. 5, 483-494.

\bibitem{FR06} B. Farkas and S. Revesz, {\it Tiles with no spectra in dimension 4},
Math. Scand. \textbf{98} (2006), no. 1, 44-52.

\bibitem{FG68} J. Feldman and F. P. Greenleaf, {\it Existence of Borel transversals in
groups.} Pacific J. Math., 25:455--461, 1968.

\bibitem{Folland-book} G. B. Folland, {\it  A Course in Abstract Harmonic Analysis},
Studies in Advance Mathematics, CRC Press, 1995.

 \bibitem{Fue74}
   B. Fuglede: {\it Commuting self-adjoint partial differential operators and a group theoretic
problem}, J. Funct. Anal. 16 (1974), 101--121. MR 57:10500


\bibitem{KG} K. Gr\"{o}chenig, {\it Aspects of Gabor analysis on locally compact abelian groups},
Gabor Analysis and Algorithms.
Applied and Numerical Harmonic Analysis 1998, pp 211--231


\bibitem{HSWW}
  E. Hern\'{a}ndez, H. Sikic, G. Weiss, E. Wilson,
  {\it Cyclic subspaces for unitary representation of LCA groups: generalized Zak transforms}.
Colloq. Math. 118 (2010), no. 1, 313 -- 332.


 \bibitem{Hewitt-Ross-I}
E. Hewitt, K. A. Ross, {\it Abstract harmonic analysis. Vol. I:
Structure of topological groups, integration theory, group
representations}. Second edition. Springer-Verlag, Berlin-New York,
(1979).


\bibitem{Iosevich-Expository} Alex Iosevich, {\it Fuglede Conjecture for Lattices}, preprint available at
\href{http://www.math.rochester.edu/people/faculty/iosevich/expository/FugledeLattice.pdf}{\footnotesize \tt www.math.rochester.edu/people/faculty/iosevich/expository/FugledeLattice.pdf}

\bibitem{IKT03} A. Iosevich, N. Katz and T. Tao,
{\it The Fuglede spectral conjecture holds for convex planar domains},
Math. Res. Lett. \textbf{10} (2003), no. 5-6, 559-569.

\bibitem{IMP15} A. Iosevich, A. Mayeli, J. Pakianathan,
{\it The Fuglede Conjecture holds in ${\Bbb Z}_p \times {\Bbb Z}_p$}, submitted.

\bibitem{IP98} A. Iosevich and S. Pedersen, {\it Spectral and tiling properties of the unit cube},
Internat. Math. Res. Notices (1998), no. 16, 819-828.


\bibitem{KL03} S. Konyagin and I. Laba,
{\it Spectra of certain types of polynomials and tiling of integers with translates of finite sets},
J. Number Theory \textbf{103} (2003), no. 2, 267-280.

 \bibitem{KL04} M. Kolountzakis and I. Laba, {\it Tiling and spectral properties of near-cubic domains},
  Studia Math. \textbf{160} (2004), no. 3, 287-299.

\bibitem{KM06} M. Kolountzakis and M. Matolcsi, {\it Tiles with no spectra},
Forum Math. \textbf{18} (2006), no. 3, 519-528.



\bibitem{L01} I. Laba, {\it Fuglede's conjecture for a union of two intervals},
Proc. Amer. Math. Soc. \textbf{129} (2001), no. 10, 2965-2972.


\bibitem{L02} I. Laba, {\it The spectral set conjecture and multiplicative properties of roots of polynomials}, J. London Math. Soc. (2) \textbf{65} (2002), no. 3, 661-671.


\bibitem{LRW00} J. Lagarias, J. Reed and Y. Wang,
{\it Orthonormal bases of exponentials for the $n$-cube}, Duke Math. J. \textbf{103} (2000), no. 1, 25-37.

\bibitem{Pedersen} S. Pedersen, {\it Spectral Theory of Commuting Self-Adjoint
Partial Differential Operators}, Journal of Functional Analysis 73,
 122-134 (1987).


\bibitem{Pontryagin} L.S. Pontryagin,  {\it Topological Groups},
 Princeton Univ. Press (1946) (Translated from Russian)

\bibitem{RS68} H. Reiter, J.D. Stegeman, {\it Classical Harmonic
Analysis on Locally Compact groups}, Clarendon Press, Oxford,
(2000).


 \bibitem{Rudin-FA} W. Rudin,  {\it Fourier Analysis on Groups}, John Wiley \& Sons, Jan 25, 1990.


 \bibitem{T04} T. Tao, {\it Fuglede's conjecture is false in 5 and higher dimensions},
 Math. Res. Lett. \textbf{11} (2004), no. 2-3, 251-258.

 \end{thebibliography}
  \end{document}